\documentclass[11pt]{amsart}
\usepackage{amssymb,bm,mathrsfs}
\usepackage{mathabx}
\usepackage[all]{xy}
\newcommand{\map}[1]{\xrightarrow{#1}}

\newcommand{\iso}{\cong}
\newcommand{\ord}{\mathrm{ord}}
\newcommand{\define}{\stackrel{\mathrm{def}}{=}}

\newcommand{\Hom}{\mathrm{Hom}}

\newcommand{\End}{\mathrm{End}}
\newcommand{\Spec}{\mathrm{Spec}}
\newcommand{\Spf}{\mathrm{Spf}}
\newcommand{\Q}{\mathbb Q}
\newcommand{\Z}{\mathbb Z}

\newcommand{\F}{\mathbb F}

\newcommand{\co}{\mathcal O}

\newcommand{\Lie}{\mathrm{Lie}}
\newcommand{\Fil}{\mathrm{Fil}}

\newcommand{\kk}{{\bm{k}}}

\newcommand{\Exc}{\mathrm{Exc}}
\newcommand{\old}{\mathrm{old}}
\newcommand{\rig}{\mathrm{rig}}

\begin{document}
\author{Benjamin Howard}
\title[Linear invariance of intersections]{Linear invariance of intersections on unitary Rapoport-Zink spaces}
\date{\today}

\thanks{This research was supported in part by NSF grants DMS1501583 and DMS1801905.}
\address{Department of Mathematics\\Boston College\\ 140 Commonwealth Ave. \\Chestnut Hill, MA 02467, USA}
\email{howardbe@bc.edu}

\begin{abstract}
We prove an invariance property of intersections of Kudla-Rapoport divisors on a unitary Rapoport-Zink space.
\end{abstract}

\maketitle

\theoremstyle{plain}
\newtheorem{theorem}{Theorem}[section]
\newtheorem{bigtheorem}{Theorem}[section]
\newtheorem{proposition}[theorem]{Proposition}
\newtheorem{lemma}[theorem]{Lemma}
\newtheorem{corollary}[theorem]{Corollary}
\newtheorem{bigcorollary}[bigtheorem]{Corollary}

\theoremstyle{definition}
\newtheorem{definition}[theorem]{Definition}
\newtheorem{hypothesis}[theorem]{Hypothesis}
\newtheorem{bighypothesis}[bigtheorem]{Hypothesis}

\theoremstyle{remark}
\newtheorem{remark}[theorem]{Remark}
\newtheorem{example}[theorem]{Example}
\newtheorem{question}[theorem]{Question}

\numberwithin{equation}{section}
\renewcommand{\thebigtheorem}{\Alph{bigtheorem}}
\renewcommand{\thebighypothesis}{\Alph{bigthypothesis}}


\section{Introduction}


Let $p$ be a prime,  let $\kk$ be a quadratic extension of $\Q_p$, and let $\co_\kk \subset \kk$ be the ring of integers.
Denote by   $\breve{\kk}$  the completion of the maximal unramified extension of $\kk$, 
let $\breve{\co}_{\kk} \subset \breve{\kk}$ be the ring of integers, 
and let $\breve{\mathfrak{m}} \subset \breve{\co}_\kk$ be the maximal ideal.  
The nontrivial automorphism of $\kk$ is denoted by  $\alpha\mapsto \overline{\alpha}$, and we  denote by 
\[
\varphi ,\overline{\varphi} : \co_\kk \to \breve{\co}_\kk
\]
the inclusion and its conjugate $\overline{\varphi}(\alpha)=\varphi(\overline{\alpha})$,  respectively.

\begin{bighypothesis}\label{hyp:unramified}
Throughout the paper we assume that either $\kk/\Q_p$ is unramified, or that $\kk/\Q_p$ is ramified but $p>2$.
\end{bighypothesis}

In this paper, we study the intersections of special divisors on a regular $n$-dimensional  Rapoport-Zink formal scheme
\[
M = M_{(1,0)} \times_{\Spf ( \breve{\co}_\kk) }  M_{(n-1,1)},
\]
flat over $\Spf ( \breve{\co}_\kk)$.  
We have imposed Hypothesis \ref{hyp:unramified}  because it is assumed in \cite{pappas} and \cite{kramer}, the results of which are needed to prove the flatness and regularity of  $M$.

The construction of $M$ depends on the choices of supersingular $p$-divisible groups $\bm{X}_0$ and $\bm{X}$ of dimensions $1$ and $n\ge 2$, respectively,  defined over the residue field $\breve{\co}_\kk / \breve{\mathfrak{m}}$ and endowed with principal polarizations and actions of $\co_\kk$.
The induced actions of $\co_\kk$ on the Lie algebras $\Lie(\bm{X}_0)$ and $\Lie(\bm{X})$ are required to satisfy signature conditions of type $(1,0)$ and $(n-1,1)$ respectively.

The precise assumptions on $\bm{X}_0$ and $\bm{X}$, along with the definition of $M$,  are explained  in \S \ref{s:RZ}. 
We note here only that the signature condition on $\bm{X}$ consists of the extra data of a codimension one subspace
$
F_{\bm{X}} \subset \Lie(\bm{X})
$
as in the work of Kr\"amer \cite{kramer}.  In particular, when $\kk/\Q_p$ is ramified our formal scheme $M_{(n-1,1)}$ does not agree with the one considered in \cite{RTW}.

As in \cite{KR-unitary-1}, the $n$-dimensional $\kk$-vector space 
\[
V = \Hom_{\co_\kk}(\bm{X}_0, \bm{X}) [1/p]
\]
 carries a natural hermitian form, and every nonzero vector $x\in V$ determines a \emph{Kudla-Rapoport divisor}  $Z(x) \subset M$; see Definition \ref{def:KR divisor}.
Our  main result concerns  arbitrary intersections of Kudla-Rapoport divisors,  including self-intersections.

For any nonzero $x\in V$, let $I_{Z(x)} \subset \co_M$ be the ideal sheaf defining $Z(x)$, and define a chain complex of locally free $\co_M$-modules
\[
C(x) = (   \cdots \to 0 \to I_{Z(x)}  \to \co_M\to 0  )
\]
supported in degrees $1$ and $0$.  We extend the definition to $x=0$ by setting 
\[
C(0) = (   \cdots \to 0 \to \omega  \map{0} \co_M\to 0  ),
\]
where $\omega$ is the line bundle of modular forms on $M$ of Definition \ref{def:modular forms}.  
This line bundle controls  the deformation theory of the Kudla-Rapoport divisors, in a sense made (somewhat) more precise in \S \ref{s:deformation}.

The following is our main result.  It is stated in the text as Theorem \ref{thm:linear invariance}.

\begin{bigtheorem}\label{main result}
Fix an $r \ge 0$, and  suppose $x_1,\ldots, x_r \in V$ and $y_1,\ldots,y_r \in V$ generate the same $\co_\kk$-submodule.
For every $i\ge 0$ there is an isomorphism of  coherent  $\co_M$-modules 
\[
 H_i ( C(x_1) \otimes \cdots \otimes C(x_r) )   \iso  H_i ( C(y_1) \otimes \cdots \otimes C(y_r) )  .
\]
\end{bigtheorem}

We can restate our main result in terms of  the Grothendieck group of coherent sheaves on $M$. 
Let  $K_0'(M)$ be the free abelian group  generated by symbols $[F]$ as $F$ runs over all isomorphism classes of coherent $\co_M$-modules, subject to the relations 
$[F_1]+[F_3] = [F_2]$ whenever there is a short exact sequence
\[
0 \to F_1\to F_2\to F_3 \to 0.
\]

In particular,  any bounded chain complex $F$  of coherent $\co_M$-modules defines a class
$
[F] = \sum_i (-1)^i \cdot [ H_i( F )]  \in K_0'(M),
$
allowing us to form 
\begin{equation}\label{grothendieck tensor}
[ C(x_1) \otimes \cdots \otimes C(x_r)] \in K_0'(M)
\end{equation}
for any finite list of vectors $x_1,\ldots, x_r\in V$.
If all $x_1,\ldots, x_r$ are nonzero, then
\[
[ C(x_1) \otimes \cdots \otimes C(x_r)] = [ \co_{Z(x_1)} \otimes^{\mathbf{L}} \cdots \otimes^{\mathbf{L}} \co_{ Z(x_r)} ] ,
\]
and hence one should regard (\ref{grothendieck tensor}) as a  generalized intersection of divisors.
On the right-hand side,  by slight abuse of notation, we are using the pushforward via $Z(x_i) \hookrightarrow M$ to  view  $\co_{Z(x_i)}$ as a coherent sheaf on $\co_M$, and $\otimes^{\mathbf{L}}$ is the derived tensor product of coherent $\co_M$-modules.

The following   is an immediate consequence of Theorem \ref{main result}.

\begin{bigcorollary}\label{bigcorGrothendieck}
If  $x_1,\ldots, x_r \in V$ and $y_1,\ldots, y_r \in V$ generate the same $\co_\kk$-submodule, then
\[
[ C(x_1)\otimes\cdots \otimes C(x_r) ]  = [ C(y_1)\otimes\cdots \otimes C(y_r) ] .
\]
\end{bigcorollary}

Perhaps the most interesting aspect of Corollary \ref{bigcorGrothendieck} is that it encodes nontrivial information about self-intersections of Kudla-Rapoport divisors.
To spell this out in the simplest  case, note that Corollary \ref{bigcorGrothendieck} implies
\begin{equation}\label{adjunction 1}
[ C(x) \otimes C(x) ] = [ C(x) \otimes  C(0)]
\end{equation}
for any  nonzero $x\in V$.  
The right hand side is  the alternating sum in $K_0'(M)$ of the homology of the complex
\[
\cdots \to 0 \to I_{Z(x)} \otimes \omega \map{\partial_2} I_{Z(x)} \oplus \omega \map{\partial_1} \co_M \to 0,
\]
where $\partial_2( a\otimes b) = ( 0,  ab )$ and  $\partial_1( a,b) =a$, and so
\[
[ C(x) \otimes C(0) ] = [ \co_M / I_{Z(x)} ] - [\omega/ I_{Z(x)} \omega ].
\]
If we again use pushforward via $ Z(x) \hookrightarrow M$ to view coherent $\co_{Z(x)}$-modules as coherent $\co_M$-modules, then (\ref{adjunction 1}) can  be rewritten as a self-intersection formula
\begin{equation}\label{adjunction 2}
[ \co_{Z(x)} \otimes^{\mathbf{L}}  \co_{ Z(x)} ] = [  \co_{Z(x)}] - [ \omega|_{ Z(x) } ] .
\end{equation}

 Because of the close connection between Grothendieck groups of coherent sheaves and Chow groups, as detailed in \cite[Chapter I]{soule92}, the global analogue of Corollary \ref{bigcorGrothendieck} has applications to conjectures of Kudla-Rapoport \cite{KR-unitary-2} on the intersection multiplicities of cycles on unitary Shimura varieties, and their connection  to derivatives of Eisenstein series.  This will be explored in forthcoming work of the author.

The formal $\breve{\co}_\kk$-scheme $M$ is locally formally of finite type, but has countably many connected components, each of which is a countable union of irreducible components.  
Let us fix one connected component $M^\circ \subset M$, and set $Z^\circ(x) = Z(x)|_{M^\circ}$.   
The following   is an immediate consequence of Theorem \ref{main result}.

\begin{bigcorollary}\label{bigcorMultiplicity}
Suppose   $x_1,\ldots, x_n \in V$ is a $\kk$-basis.  The Serre intersection multiplicity
\begin{eqnarray}\lefteqn{
\chi\big(    \co_{Z^\circ(x_1)}  \otimes^{\mathbf{L}} \cdots \otimes^{\mathbf{L}}  \co_{Z^\circ(x_n)}  \big) } \nonumber  \\
& \define &\sum_{i,j \ge 0} (-1)^{ i+j } \, 
\mathrm{length}_{\breve{\co}_\kk} H^j \big( M^\circ ,   H_i\big(    \co_{Z^\circ(x_1)}  \otimes^{\mathbf{L}} \cdots \otimes^{\mathbf{L}}  \co_{Z^\circ(x_n)}  \big)   \big) \nonumber
\end{eqnarray}
depends only on the $\co_\kk$-lattice spanned by  $x_1,\ldots, x_n$.
\end{bigcorollary}

It is conjectured by Kudla-Rapoport that the intersection multiplicity appearing in Corollary \ref{bigcorMultiplicity} is related to derivatives of representation densities.  
When $\kk/\Q_p$ is unramified  this is  \cite[Conjecture 1.3]{KR-unitary-1}.  
When $\kk/\Q_p$ is ramified it is  perhaps not clear what the precise statement of the conjecture should be.

 
 
 \bigskip \noindent \textbf{Relation to previous results.}
Weaker versions of the results stated above can be proved using a  simpler argument\footnote{Terstiege only considers the case $n=3$, but his argument generalizes to all $n$.}  of Terstiege   \cite{terstiege}.    
We  clarify here what one can and cannot prove using that argument.

When $\kk/\Q_p$ is unramified,  Corollary \ref{bigcorMultiplicity} is  \cite[Proposition 3.2]{terstiege}.  
Terstiege's argument  can also be used to prove Theorem \ref{main result} and Corollary \ref{bigcorGrothendieck}, but only under the additional assumption that the vectors $x_1,\ldots, x_r$ (equivalently, $y_1,\ldots, y_r$) are linearly independent.  
 In particular, his argument does not give self-intersection formulas like (\ref{adjunction 1}) and (\ref{adjunction 2}).

The key thing that makes  Terstiege's argument work is that,  in the unramified case, the Kudla-Rapoport divisors $Z(x)$ and $Z(x')$ defined by linearly independent vectors $x,x'\in V$ are flat over $\breve{\co}_\kk$,  from which it follows that their intersection $Z(x) \cap Z(x')$ lies in   codimension $2$.

When $\kk/\Q_p$ is ramified  the situation is very different: the Kudla-Rapoport divisors are usually not flat, and the  intersection  $Z(x) \cap Z(x')$ is often  of codimension $1$.    In fact, it is easy to see using Proposition \ref{prop:lots of exceptional} that one can  construct a basis $x_1,\ldots, x_n\in V$  and an effective  Cartier divisor $D \subset M$,  contained in the special fiber (in the sense that the structure sheaf $\co_D$ is annihilated by a uniformizer in $\co_\kk$),  such that 
 \[
 D\subset Z(x_1) \cap \cdots \cap Z(x_n) .
 \]
Because of this, the argument used by Terstiege breaks down in a fundamental way when $\kk/\Q_p$ is ramified, and seems to yield little information in the direction of Theorem \ref{main result} and its corollaries.

 
 
 \bigskip \noindent \textbf{The strategy of the proof.}
 To explain the key idea underlying the proof of Theorem \ref{main result},  suppose we have  vectors $x_1,x_2,y_1,y_2 \in V$ related by 
\[
 y_1 = x_1 + ax_2 ,\qquad  y_2  = x_2
\]
 for some $a\in \co_\kk$.   In particular, $\{x_1,x_2\}$ and $\{y_1,y_2\}$ generate the same $\co_\kk$-submodule of $V$.

 One should imagine that there are global sections 
 \begin{equation}\label{pretend sections}
 s_1,s_2,t_1,t_2 \in H^0( M , \omega^{-1}) 
 \end{equation}
 satisfying $ \mathrm{div}(s_i) = Z(x_i)$ and  $\mathrm{div}(t_i) = Z(y_i)$, and also satisfying 
\begin{equation}\label{linear sections}
 t_1 = s_1 + a s_2 ,\qquad
 t_2  = s_2.
\end{equation}
Such sections  would determine complexes
\begin{align*}
D(x_i)  &= (   \cdots \to 0 \to \omega  \map{s_i} \co_M\to 0  )  \\
D(y_i)  &= (   \cdots \to 0 \to \omega  \map{t_i} \co_M\to 0  ) ,
\end{align*}
along with canonical isomorphisms
\[
C(x_i) \iso D(x_i) ,\qquad C(y_i) \iso D(y_i).
\]
 Indeed,  if $x_i \neq 0$ then 
\begin{equation*}
\xymatrix{
{\cdots}\ar[r]  & 0 \ar[r] \ar@{=}[d] & {\omega} \ar[r]^{ s_i}  \ar[d]^{s_i} & { \co_M} \ar[r] \ar@{=}[d]& 0  \\
{\cdots}\ar[r]   & 0 \ar[r] & { I_{Z(x_i) } }  \ar[r] & { \co_M } \ar[r]& 0
}
\end{equation*}
defines  an isomorphism $D(x_i)\iso C(x_i)$.   If $x_i=0$ then $s_i=0$,  and   $C(x_i)$ and $D(x_i)$  are equal simply by definition.
The point of replacing the complexes $C(\cdot)$ by the isomorphic complexes $D(\cdot)$  is that the relations 
 (\ref{linear sections}) induce relations amongst the $D(\cdot)$, which allow one  to write down (see the proof of Lemma \ref{lem:wee complex}) an explicit isomorphism
 \[
 D(x_1)\otimes D(x_2) \iso D(y_1)\otimes D(y_2).
 \]
In this way one would obtain from (\ref{pretend sections}) an isomorphism of complexes
  \begin{equation}\label{pretend complexes}
 C(x_1)\otimes C(x_2) \iso C(y_1)\otimes C(y_2).
 \end{equation}
 
 Unfortunately,  sections (\ref{pretend sections}) with the required properties need not exist globally on $M$, and so neither does the isomorphism (\ref{pretend complexes}).
 Instead, our approach is to use Grothendieck-Messing theory to construct sections $s_i$ and $t_i$  defined only on the first order infinitesimal neighborhoods of $Z(x_i)$ and $Z(y_i)$ in $M$.   
 Working on a sufficiently fine Zariski open cover $\mathcal{U}$ of $M$, we then choose local approximations of these sections, and so obtain, by the method above, an isomorphism 
   \begin{equation}\label{local pretend complexes}
  C(x_1)_{U}\otimes C(x_2)_{U} \iso C(y_1)_{U}\otimes C(y_2)_{U}
 \end{equation}
 over each $U\in \mathcal{U}$.  
 Because there is no canonical way to choose these local approximations, the isomorphisms (\ref{local pretend complexes}) need not glue together as $U\in \mathcal{U}$ varies.
However,  if one imposes  mild restrictions on the local approximations, the homotopy class of  (\ref{local pretend complexes}) is independent of the choices. 
  The resulting isomorphisms
 \[
  H_i(  C(x_1) \otimes C(x_2) ) _U   \iso  H_i ( C(y_1)   \otimes C(y_2) ) _U 
 \]
of $\co_U$-modules can therefore be glued together as $U\in \mathcal{U}$ varies.


\section{The Rapoport-Zink space and its divisors}
\label{s:RZ}


Fix a triple $(\bm{X}_0 , \bm{i}_0 , \bm{\lambda_0 } )$ in which
\begin{itemize}
\item
 $\bm{X}_0$ is a supersingular $p$-divisible group over $\breve{\co}_\kk/\breve{\mathfrak{m}}$ of  dimension $1$,
\item
$\bm{i}_0 : \co_\kk \to \End(\bm{X}_0)$ is an action of $\co_\kk$ on $\bm{X}_0$ such that  the induced action on $\Lie(\bm{X}_0)$ is through the inclusion $\varphi : \co_\kk \to \breve{\co}_\kk$,
\item
$\bm{\lambda}_0 : \bm{X}_0 \to \bm{X}_0^\vee$  is a principal polarization compatible with the $\co_\kk$-action, in the sense that the induced Rosati involution $\dagger$ satisfies
$
\bm{i}_0(\alpha)^\dagger = \bm{i}_0(\overline{\alpha})
$
  for all $\alpha\in \co_\kk$.
\end{itemize}

From the above data one can construct a Rapoport-Zink formal scheme by specifying its functor of points.
Let  $\mathrm{Nilp}$ the category of $\breve{\co}_\kk$-schemes on which $p$ is locally nilpotent.  
For each $S\in \mathrm{Nilp}$ let $M_{(1,0)}(S)$ be the set of isomorphism classes of quadruples
$(X_0,i_0,\lambda_0 ,\varrho_0)$ in which 
\begin{itemize}
\item
$X_0$ is a $p$-divisible group  over $S$ of dimension $1$, 
\item
$i_0 : \co_\kk \to \End(X_0)$ is an action of $\co_\kk$ on $X_0$ such that  the induced action on $\Lie(X_0)$ is through the inclusion $\varphi:\co_\kk\to \breve{\co}_\kk$,

\item
$\lambda_0 : X_0 \to X_0^\vee$  is a principal polarization compatible with $\co_\kk$-action in the sense above, 
\item
$\varrho_0 :  X_0 \times_S\overline{S} \to \bm{X}_0\times_{\Spec(\breve{\co}_\kk/\breve{\mathfrak{m}})}  \overline{S}$ is  an $\co_\kk$-linear quasi-isogeny,  respecting polarizations up to scaling by $\Q_p^\times$.
Here  
\[
\overline{S} = S \times_{ \Spec(\breve{\co}_\kk) } \Spec(\breve{\co}_\kk/\breve{\mathfrak{m}} ).
\]
\end{itemize}
An isomorphism between two such tuples is an $\co_\kk$-linear isomorphism of $p$-divisible groups 
$X_0 \iso X_0'$  identifying $\varrho_0$ with $\varrho'_0$, and identifying $\lambda_0$ with $\lambda_0'$ up to $\Z_p^\times$-scaling.

\begin{proposition}
The functor $M_{(1,0)}$ is represented by a countable disjoint union of copies of  $\Spf(\breve{\co}_\kk)$. 
 \end{proposition}

\begin{proof}
The formal deformation space of the triple $(\bm{X}_0, \bm{i}_0,\bm{\lambda}_0)$ is $\Spf(\breve{\co}_\kk)$.
This can be proved using Lubin-Tate theory.  Alternatively, it is a special case of 
 \cite[Theorem 2.1.3]{Ho-KR-1}, which applies to more general $p$-divisible groups with complex multiplication.
With this fact in mind, the proof is the same as the $d=1$ case of \cite[Proposition 3.79]{rapoport-zink}.
\end{proof}

Now fix a tuple $(\bm{X} , \bm{i} ,  \bm{\lambda} ,F_{\bm{X}})$ in which 
\begin{itemize}
\item
$\bm{X}$ is a supersingular $p$-divisible group  over $ \breve{\co}_\kk/\breve{\mathfrak{m}}$ of  dimension $n$, 
\item
$\bm{i} : \co_\kk \to \End(\bm{X})$ is an action of $\co_\kk$ on $\bm{X}$,
\item
$\bm{\lambda} : \bm{X} \to \bm{X}^\vee$ is a principal polarization compatible with the $\co_\kk$-action in the sense above,
\item
$F_{\bm{X}} \subset \Lie(\bm{X})$ is an  $\breve{\co}_\kk/\breve{\mathfrak{m}}$-module direct summand of rank $n-1$ satisfying Kr\"amer's  \cite{kramer} signature condition:
 the action of  $\co_\kk$ on $\Lie(\bm{X})$ induced by $\bm{i} : \co_\kk \to \End(\bm{X})$ stabilizes $F_{\bm{X}}$, and acts on $F_{\bm{X}}$ and $\Lie(\bm{X})/F_{\bm{X}}$  through 
$\varphi, \overline{\varphi}: \co_\kk \to \breve{\co}_\kk$,  respectively.
\end{itemize}

For each $S\in \mathrm{Nilp}$ let $M_{(n-1,1)}(S)$ be the set of isomorphism classes of tuples $(X,i,\lambda, F_X,\varrho)$ in which 
\begin{itemize}
\item
$X$ is a $p$-divisible group  over $S$ of dimension $n$, 
\item
$ i : \co_\kk \to \End(X)$ is an action of $\co_\kk$ on $X$,  
\item
$\lambda : X \to X^\vee$ is a principal polarization compatible with the $\co_\kk$-action in the sense above,
\item
$F_X \subset \Lie(X)$ is a local $\co_S$-module local direct summand  of rank $n-1$ satisfying Kr\"amer's signature condition as above,
\item
 $\varrho :  X\times_S {\overline{S}}  \to \bm{X} \times_{  \Spec(\breve{\co}_\kk / \breve{\mathfrak{m}} )  }  \overline{S}$ is an $\co_\kk$-linear quasi-isogeny respecting polarizations up to scaling by $\Q_p^\times$.
 \end{itemize}
An isomorphism between two such tuples is an $\co_\kk$-linear isomorphism of $p$-divisible groups 
$X \iso X'$  identifying $F_X$ with  $F_{X'}$, identifying $\varrho$ with $\varrho'$, and identifying $\lambda$ with $\lambda'$ up to  $\Z_p^\times$-scaling.

\begin{proposition}
The functor $M_{(n-1,1)}$ is represented by a formal $\breve{\co}_\kk$-scheme, locally formally of finite type.
Moreover
\begin{enumerate}
\item
$M_{(n-1,1)}$ is flat over $\breve{\co}_\kk$, and regular of dimension $n$;
\item
if $\kk/\Q_p$ is unramified then $M$ is formally smooth over $\breve{\co}_\kk$.
\end{enumerate}
\end{proposition}

\begin{proof}
First suppose that $p>2$.
The representability follows from the general results of Rapoport-Zink   \cite[Theorem 3.25]{rapoport-zink}.
The remaining claims can be verified using the theory of local models, as  in \cite{pappas} and \cite[Proposition 3.33]{rapoport-zink}.
In the unramified case the analysis of the local model is routine, and in the ramified case it was done by Kr\"amer \cite{kramer}.

The $p=2$ case is excluded from much of \cite{rapoport-zink} by the blanket assumption imposed in \cite[p.~75]{rapoport-zink}, and the author  is unaware of a published or publicly available reference for this case\footnote{When $p=2$ there is a thorough study of unitary Rapoport-Zink spaces of  signature $(1,1)$  in the work of Kirch \cite{kirch}, even when $\kk/\Q_p$ is ramified.
}.
However, M.~Rapoport has informed the author that  the necessary extensions to  $p=2$ with $\kk/\Q_p$ unramified  will appear in an appendix to the forthcoming work \cite{RSZ}.  
\end{proof}

Following \cite{KR-unitary-1},  we will define a family of divisors on 
\begin{equation*}
M = M_{(1,0)} \times_{\Spf( \breve{\co}_\kk) } M_{(n-1,1)}.
\end{equation*}
If $S\in \mathrm{Nilp}$, we will write $S$-points of $M$ simply as
$
(X_0,X) \in M(S),
$
rather than the cumbersome $(X_0,i_0,\lambda_0,\varrho_0,X,i,\lambda, F_X, \varrho)$.

\begin{lemma}\label{lem:hermitian inclusion}
The $\kk$-vector space
\[
V = \Hom_{\co_\kk}(\bm{X}_0 , \bm{X}) [1/p]
\]
has dimension $n$. 
For any  $S\in \mathrm{Nilp}$ and  any  $(X_0,X) \in M(S)$  there is a canonical inclusion\footnote{Here one must interpret the right hand side as  global sections of the Zariski sheaf
$\underline{\Hom}( X_0, X ) [1/p]$ on $S$, as in \cite[Definition 2.8]{rapoport-zink}.  If $S$ is quasi-compact this agrees with the naive definition.  We will ignore this technical point in all that follows.}
\begin{equation}\label{hermitian inclusion}
V\subset  \Hom_{\co_\kk}( X_0 , X ) [1/p].
\end{equation}
\end{lemma}

\begin{proof}
As $\bm{X}$ is supersingular, there is a quasi-isogeny of $p$-divisible groups
\[
\bm{X} \to \bm{X}_0 \times \cdots \times \bm{X}_0.
\]
The Noether-Skolem theorem implies that any two embeddings of $\kk$ into 
\[
\End(\bm{X}) [1/p] \iso M_n( \End(\bm{X}_0)) [1/p] 
\]
are conjugate, and hence this quasi-isogeny can be chosen to be $\co_\kk$-linear.
It follows that
\[
V\iso \End_{\co_\kk}(\bm{X}_0)[1/p] \times \cdots \times \End_{\co_\kk}(\bm{X}_0)[1/p] .
\]
 Each factor on the right has dimension one, proving the first claim of the lemma.

Given $x\in V$, the quasi-isogenies $\varrho_0$ and $\varrho$ allow us to identify $x$ with 
\[
  \varrho^{-1}\circ x \circ  \varrho_0 \in  \Hom_{\co_\kk}(X_0 \times_S \overline{S}, X \times_S \overline{S} ) [1/p] .
\]
The reduction map 
\[
\Hom_{\co_\kk}( X_0 , X ) [1/p] \to   \Hom_{\co_\kk}(X_0 \times_S \overline{S}, X \times_S \overline{S} ) [1/p] 
\]
is an isomorphism by \cite[Lemma 1.1.3]{katz}, proving the second claim of the lemma.
\end{proof}

The second claim of Lemma \ref{lem:hermitian inclusion} allows us to make the following definition.

 \begin{definition}\label{def:KR divisor}
 For any nonzero $x\in V$ we define the  \emph{Kudla-Rapoport divisor} to be the closed formal subscheme  
 \[
 Z(x)\subset M
 \]
 whose functor of points assigns to any $S\in \mathrm{Nilp}$  the set of  all $(X_0,X) \in M(S)$ for which
$x\in  \Hom_{\co_\kk}( X_0 , X )$ under the inclusion (\ref{hermitian inclusion}).
  \end{definition}

 When $\kk/\Q_p$ is unramified, it is proved in \cite{KR-unitary-1} that $Z(x) \subset M$ is defined locally by a single equation.  A proof of the same claim  in the ramified case can be found in \cite{Ho-KR-2}.  
We will reprove these results below in Proposition \ref{prop:cartier},  as the arguments provide additional information that will be  essential for the proof of Theorem \ref{thm:linear invariance}.


\section{Vector bundles}
\label{s:vector bundles}


For the remainder of the paper $(X_0, X)$  denotes the universal object over 
\[
M = M_{(1,0)} \times_{\Spf( \breve{\co}_\kk) } M_{(n-1,1)}.
\]
Let  $D(X)$ be  the restriction to the Zariski site of the covariant Grothendieck-Messing crystal of  $X$.  
Thus $D(X)$ is a vector bundle  on $M$ of rank $2n$, sitting in a short exact sequence
\[
0 \to \Fil (X) \to D(X) \to \Lie(X) \to 0.
\]
Similarly, the Grothendieck-Messing crystal of $X_0$ determines  a  short exact sequence
\[
0 \to \Fil (X_0) \to D(X_0) \to \Lie(X_0) \to 0,
\]
of vector bundles on $M$.   

 The actions $i_0 :\co_\kk\to \End(X_0)$ and $i : \co_\kk \to \End(X)$ induce actions of $\co_\kk$ on all of these vector bundles, and the above short exact sequences are $\co_\kk$-linear.
The principal polarization on $X$ induces a perfect alternating pairing
\[
\langle\cdot,\cdot\rangle : D(X)\times D(X) \to \co_M,
\]
which is compatible with the action $i:\co_\kk \to \End_{\co_M}( D(X))$,  in the sense that
\begin{equation}\label{hermitian}
\langle i(\alpha)  x,y \rangle = \langle x, i(\overline{\alpha}) y \rangle
\end{equation}
for all $\alpha\in \co_\kk$ and all local sections $x$ and $y$ of $D(X)$.  The local direct summand $\Fil(X) \subset D(X)$ is maximal isotropic with respect to this pairing, and hence there is an induced perfect pairing
\begin{equation}\label{isotropy pairing}
\langle\cdot,\cdot\rangle : \Fil(X) \times \Lie(X) \to \co_M.
\end{equation}

By virtue of the moduli problem defining $M_{(n-1,1)}$, there is a distinguished local direct summand 
$
F_X \subset \Lie(X)
$
 of rank $n-1$, whose annihilator with respect to the pairing (\ref{isotropy pairing}) is a  local direct summand
$
F_X^\perp \subset \Fil(X)
$
of rank one.
Both submodules are stable under the  action of $\co_\kk$, which acts
\begin{itemize}
\item
on  $F_X$ and $F_X^\perp$ via $\varphi :\co_\kk \to \breve{\co}_\kk$,
\item
on $\Lie(X)/F_X$ and  $\Fil(X)/F^\perp_X$  via  $\overline{\varphi}:\co_\kk \to \breve{\co}_\kk$.
\end{itemize}

There is a natural morphism of $\co_M$-algebras
\[
\co_\kk \otimes_{\Z_p} \co_M \map{   \alpha\otimes 1\mapsto ( \varphi(\alpha) , \overline{\varphi}(\alpha))   } 
 \co_M \times \co_M.
\]
If  $\kk/\Q_p$ is unramified this map is an isomorphism, and we obtain a pair of orthogonal idempotents in $\co_\kk \otimes_{\Z_p} \co_M$.  
Without any assumption on ramification, one can still define reasonable substitutes for these idempotents.  To do so, fix a $\beta \in \co_\kk$ satisfying $\co_\kk = \Z_p + \Z_p\beta$, and define
\begin{align*}
\epsilon   & = \beta \otimes 1 - 1 \otimes \overline{\varphi} (\beta) \in  \co_\kk \otimes_{\Z_p} \co_M \\
\overline{\epsilon}   & =  \overline{\beta} \otimes 1 - 1\otimes \overline{\varphi}( \beta ) \in  \co_\kk \otimes_{\Z_p} \co_M .
\end{align*}
The ideal sheaves in $\co_\kk\otimes_{\Z_p}\co_M$ generated by these elements are independent of the choice of $\beta$, and there are short exact sequences of $\co_M$-modules
\begin{align*}
0 \to (\epsilon) \to  \co_\kk \otimes_{\Z_p}  \co_M \map{\alpha \otimes 1 \mapsto \overline{\varphi}(\alpha) }  \co_M \to 0 \\
0 \to ( \overline{\epsilon} ) \to  \co_\kk \otimes_{\Z_p} \co_M \map{\alpha \otimes 1 \mapsto \varphi(\alpha) }  \co_M \to 0 .
\end{align*}

\begin{remark}\label{rem:direct summand}
In particular, $(\epsilon)$ and $(\overline{\epsilon})$  are  rank one  $\co_M$-module  local direct summands of $\co_\kk\otimes_{\Z_p}\co_M$.
\end{remark}

Let $\mathfrak{d} \subset \co_\kk$ be the different of $\kk/\Q_p$, and set $\breve{\mathfrak{d}}= \varphi(\mathfrak{d}) \breve{\co}_\kk.$  It follows from  Hypothesis \ref{hyp:unramified} that
\begin{equation*}
\breve{\mathfrak{d}}  = \begin{cases}
 \breve{\co}_\kk  &\mbox{if }  \kk/\Q_p \mbox{ is unramified} \\
  \breve{\mathfrak{m}}  &\mbox{if }  \kk/\Q_p \mbox{ is ramified.} 
  \end{cases}
\end{equation*}

\begin{lemma}\label{lem:idempotents}
 Suppose $N$ is an $\co_M$-module endowed with an action 
 \[
 i  : \co_\kk \to \End_{ \co_M  }  (N) .
 \]
   If we view $N$ as an $\co_\kk \otimes_{\Z_p}\co_M$-module, then  $N/ \overline{\epsilon} N$ and $N/\epsilon N$  are the maximal quotients of $N$ on which $\co_\kk$ acts through $\varphi$ and $\overline{\varphi}$, respectively.
Moreover, 
\begin{align*}
\epsilon N & \subset  \{ n\in N  : \forall \alpha\in \co_\kk, \  i  (\alpha)  x  = \varphi(\alpha)x  \} \\
 \overline{\epsilon} N  & \subset\{ n \in N : \forall \alpha\in \co_\kk, \  i  (\alpha) x = \overline{\varphi}(\alpha)x  \} ,
\end{align*}
and both quotients are annihilated by $\breve{\mathfrak{d}} \co_M$.
 \end{lemma}
 
 \begin{proof}
This is an elementary exercise, left to the reader.
\end{proof}
 
\begin{proposition}\label{prop:new line}
There are inclusions of $\co_M$-module local direct summands
$
F_X^\perp \subset \epsilon D(X) \subset D(X).
$
The morphism $\epsilon : D(X) \to \epsilon D(X)$  descends to a surjection
\begin{equation}\label{big quotient}
\Lie(X) \map{\epsilon} \epsilon D(X) / F_X^\perp,
\end{equation}
whose kernel 
$
L_X \subset \Lie(X)
$
 is an $\co_M$-module local direct summand of rank one.  
It is stable under $\co_\kk$, which acts on $\Lie(X)/L_X$   and  $L_X$  via $\varphi, \overline{\varphi} :\co_\kk \to \breve{\co}_\kk$, respectively.
\end{proposition}

 \begin{proof}
The vector bundle $D(X)$ is locally free of rank $n$ over $\co_\kk\otimes_{\Z_p}\co_M$, and hence 
$\epsilon D(X) \subset D(X)$ is a local $\co_M$-module direct summand by Remark \ref{rem:direct summand}.
As $F_X$ is locally free over $\co_M$, the perfect pairing
\[
\big( \Fil(X) / F_X^\perp\big) \otimes F_X  \to   \co_M,
\]  
induced by (\ref{isotropy pairing}) shows that  $\Fil(X)/ F_X^\perp$ is locally free, from which it follows that $F_X^\perp$ a local direct summand of  $D(X)$.  

Now consider the perfect pairing 
\[
F_X^\perp \otimes (\Lie(X) / F_X) \to \co_M
\]
induced by (\ref{isotropy pairing}).  As $\co_\kk$ acts on  $\Lie(X)/F_X$ via $\overline{\varphi}$, the relation (\ref{hermitian})  implies that $\co_\kk$ acts on $F_X^\perp$ via $\varphi$.  
Lemma \ref{lem:idempotents} thus implies 
\[
\breve{\mathfrak{d}} F_X^\perp \subset \epsilon   F_X^\perp \subset F_X^\perp,
\]
and so $\breve{\mathfrak{d}} F_X^\perp \subset \epsilon D(X)$.  
The stronger inclusion $F_X^\perp \subset \epsilon D(X)$ then follows from the fact that $D(X)/\epsilon D(X)$ is $\co_M$-torsion free.

As $\co_\kk$ acts on $F_X$ through $\varphi:\co_\kk \to \breve{\co}_\kk$, we must have $\overline{\epsilon} F_X =0$.  
Hence 
\[
\langle \epsilon x,y \rangle = \langle x, \overline{\epsilon} y\rangle =0
\]
for all local sections  $x$ and $y$ of $\Fil(X)$ and $F_X$, respectively.  Thus 
\[
\epsilon \Fil(X) \subset F_X^\perp ,
\]
and the map  (\ref{big quotient}) is well-defined.

The kernel $L_X$ of  (\ref{big quotient})  is  a local direct summand, as  (\ref{big quotient}) is a surjection to a locally free $\co_M$-module.
Moreover,  Lemma \ref{lem:idempotents} implies that $\co_\kk$ acts on the codomain  via $\varphi$, and hence acts on $\Lie(X)/L_X$ in the same way.
 
Suppose the natural map
$
L_X \to \Lie(X) / F_X
$ 
is trivial.  The inclusion $L_X \subset F_X$ then shows that  $\co_\kk$ acts on both $L_X$ and $\Lie(X)/L_X$ via $\varphi$, and hence both are annihilated by $\overline{\epsilon}$.  This means that $\overline{\epsilon}\cdot \overline{\epsilon}$ annihilates $\Lie(X)$.
But $\overline{\epsilon}$ acts on $\Lie(X)/F_X$ via the nonzero scalar $\varphi(\beta-\overline{\beta}) \in \breve{\co}_\kk$,  a contradiction.

The map  $L_X \to \Lie(X) / F_X$  is therefore nonzero, and hence injective as $M$ is locally integral.
As $\co_\kk$ acts on the codomain via $\overline{\varphi}$, it acts in the same way on $L_X$.
\end{proof}

The line bundle $L_X$ of Proposition \ref{prop:new line} is,  by construction, 
the pullback  of a line bundle on $M_{(n-1,1)}$ via the projection $M\to M_{(n-1,1)}$.  
We will now twist it  by a line bundle pulled back via  $M\to M_{(1,0)}$.

\begin{definition}\label{def:modular forms}
The \emph{line bundle of  modular forms}  $\omega$ is the invertible sheaf of $\co_M$-modules with inverse
\[
\omega^{-1} =  \underline{\Hom} ( \Fil (X_0) ,  L_X  ).
\]
\end{definition}

 \begin{remark}
The line bundle of Definition \ref{def:modular forms}  does not agree with the line bundle of modular forms defined in  \cite{BHKRY-1,BHKRY-2}.  
In those papers the line bundle of modular forms, which we here denote by $\omega_\old$,  is characterized by
\[
\omega_\old^{-1} =  \underline{\Hom} ( \Fil (X_0) , \Lie( X) / F_{X}  ) .
\]
The inclusion $L_X \subset \Lie(X)$ induces a  morphism $L_X \to  \Lie(X) / F_X$,
 which in turn induces $\omega_\old \to \omega$. 
It is not difficult to check that this latter map identifies
\[
\breve{\mathfrak{d}} \cdot \omega \subset \omega_\old \subset \omega,
\]
but when $\kk/\Q_p$ is ramified neither inclusion is an equality.  
\end{remark}


\section{Deformation theory}
\label{s:deformation}


Suppose $Z\subset M$ is any closed formal subscheme, and denote by  $I_Z\subset \co_M$ its ideal sheaf.
The square $I_{\widetilde{Z}}= I_Z^2$ is the ideal sheaf of  a larger closed formal subscheme
\[
Z\subset \widetilde{Z} \subset M
\]
called the \emph{first order infinitesimal neighborhood} of $Z$ in $M$.

Now fix a nonzero $x\in V$ and consider the first order infinitesimal neighborhood 
\[
Z(x) \subset \widetilde{Z}(x) \subset M
\]
of the corresponding Kudla-Rapoport divisor.
By the very definition (Definition \ref{def:KR divisor}) of $Z(x)$, when we restrict the universal object $(X_0,X)$ to $Z(x)$ we obtain  a distinguished morphism of $p$-divisible groups
\begin{equation*}
X_0|_{ Z(x) }   \map{x} X|_{ Z(x) }.
\end{equation*}
This induces an $\co_\kk$-linear morphism of vector bundles 
\begin{equation}\label{parallel transport}
D(X_0)|_{Z(x)} \map{x} D(X)|_{ Z(x)}
\end{equation}
on $Z(x)$, which respects the Hodge filtrations. 
By Grothendieck-Messing theory this morphism admits a canonical extension 
\[
D(X_0)|_{\widetilde{Z}(x)} \map{\widetilde{x}} D(X)|_{\widetilde{Z}(x)}
\]
to the first order infinitesimal neighborhood,  which no longer respects the Hodge filtrations. 
 Instead, it determines a nontrivial morphism
\begin{equation}\label{true obstruction}
\Fil(X_0)|_{\widetilde{Z}(x)}   \map{\widetilde{x}}  \Lie(X)|_{\widetilde{Z}(x)}.
\end{equation}

\begin{proposition}
The morphism (\ref{true obstruction}) takes values in the rank one local direct summand
\[
 L_X|_{\widetilde{Z}(x) } \subset \Lie(X)|_{\widetilde{Z}(x)},
\]
and so can be viewed as a morphism of line bundles
\begin{equation}\label{obstruction}
\Fil(X_0)|_{\widetilde{Z}(x)}   \map{\widetilde{x}}  L_X|_{\widetilde{Z}(x)}.
\end{equation}
The Kudla-Rapoport divisor $Z(x)$ is the largest closed formal subscheme of $\widetilde{Z}(x)$ over which (\ref{obstruction})  is trivial.
\end{proposition}

\begin{proof}
The vector bundle $D(X_0)$ is locally free of rank one over $\co_\kk\otimes_{\Z_p} \co_M$, and its quotient
\[
D(X_0) / \Fil(X_0) \iso \Lie(X_0)
\]
is annihilated by $\overline{\epsilon}$.  Hence $ \overline{\epsilon} \cdot D(X_0) \subset \Fil(X_0)$, 
and equality holds as both are rank one local $\co_M$-module direct summands of $D(X_0)$;
see Remark \ref{rem:direct summand}.

It follows  that  (\ref{true obstruction}) takes values in the subsheaf
\[
  \overline{\epsilon} \cdot  \Lie(X)|_{\widetilde{Z}(x)} \subset \Lie(X)|_{\widetilde{Z}(x)} .
\]
On the other hand, the final claim of Proposition \ref{prop:new line} implies that $\overline{\epsilon}$ annihilates  $\Lie(X) / L_X$, and  hence
\[
  \overline{\epsilon} \cdot  \Lie(X)|_{\widetilde{Z}(x)}  \subset L_X|_{\widetilde{Z}(x)}.
\]
This proves the first claim.

For the second claim, it follows from Grothendieck-Messing theory that $Z(x)$ is the largest closed formal subscheme of $\widetilde{Z}(x)$ along which (\ref{true obstruction}) vanishes.  
As $L_X \subset \Lie(X)$ is a local direct summand, this is equivalent to (\ref{obstruction}) vanishing.
\end{proof}

\begin{definition}\label{def:obstruction}
The section 
\[
\mathrm{obst}(x) \in H^0\big(  \widetilde{Z}(x) , \omega^{-1}|_{  \widetilde{Z}(x)  }  \big)
\]
determined by (\ref{obstruction}) is called the \emph{obstruction to deforming $x$}. 
As we have already explained,  $Z(x)$ is the largest closed formal subscheme of $\widetilde{Z}(x)$ over which $\mathrm{obst}(x)=0$.
\end{definition}

 \begin{proposition}\label{prop:cartier}
For any nonzero $x\in V$, the closed formal subscheme  $Z(x)\subset M$ is a  Cartier divisor; that is to say, it is defined locally by a single nonzero equation.
\end{proposition}
 
\begin{proof}
Let $R$ be the local ring of $M$ at a point $z\in Z(x)$, and let $I\supset I^2$
 be the ideals of $R$ corresponding to  $Z(x) \subset \widetilde{Z}(x)$.
 After pulling back via $\Spf(R) \to M$, we may trivialize the line bundle $\omega$, and  the obstruction to deforming $x$ becomes an $R$-module generator 
\[
\mathrm{obst}(x) \in I / I^2 . 
\] 
It  follows from Nakayama's lemma that $I\subset R$ is a principal ideal, and it only remains to show that $I\neq 0$.

Suppose $I=0$.  This implies that we may find an open subset $U\subset M$ such that $Z(x)|_U = U$.
As in \cite[Chapter 5]{rapoport-zink}, $M$ has an associated rigid analytic space $M^\rig$ over $\breve{\kk}$, and $U \subset M$ determines an admissible open subset
\[
U^\rig \subset M^\rig.
\]

The vector bundles of \S \ref{s:vector bundles} determine  filtered vector bundles 
\begin{align*}
  \Fil (X_0)^\rig & \subset D(X_0)^\rig \\
    \Fil (X)^\rig & \subset D(X)^\rig
\end{align*}
on $M^\rig$.  By \cite[Proposition 5.17]{rapoport-zink} these admit $\co_\kk$-linear trivializations
\begin{align}
 D(X_0)^\rig  & \iso  V_0 \otimes_{\breve{\kk}} \co_{M^\rig}  \label{trivialization 1}\\
 \qquad D(X)^\rig  & \iso  V \otimes_{\breve{\kk}} \co_{M^\rig}, \label{trivialization 2}
\end{align}
where $V_0$ and $V$ are vector spaces over  $\breve{\kk}$ of dimensions $2$ and $2n$, respectively, 
 endowed with actions $i_0 : \kk \to \End_{\breve{\kk}}(V_0)$ and  $i  : \kk \to \End_{\breve{\kk}}(V)$.


The signature $(1,0)$ condition on $X_0$ implies that $\kk$ acts on $\Fil(X_0)^\rig$ via $\overline{\varphi}:\kk \to \breve{\kk}$.  From this it follows easily that (\ref{trivialization 1}) induces an identification of line bundles
\[
   \overline{\epsilon} \cdot  \Fil (X_0)^\rig =  ( \overline{\epsilon}V_0) \otimes _{\breve{\kk}} \co_{M^\rig}.
\]
One the other hand, the signature $(n-1,1)$ condition on $X$ implies that (\ref{trivialization 2}) determines an inclusion
\[
\overline{\epsilon} \cdot  \Fil (X)^\rig \subset  ( \overline{\epsilon}V ) \otimes_{\breve{\kk}} \co_{M^\rig}
\]
as a local direct summand of corank one.  
This inclusion determines the Grothendieck-Messing  (or Gross-Hopkins) period morphism
\begin{equation}\label{periods}
\pi : M^\rig \to N^\rig
\end{equation}
to the rigid analytic flag variety $N^\rig$ parameterizing all codimension one subspaces of $\overline{\epsilon}V$.
It follows from \cite[Proposition 5.17]{rapoport-zink} that $\pi$ is \'etale.

After restriction to  $U^\rig$ the morphism (\ref{parallel transport}) determines a morphism
\[
 D(X_0)^\rig|_{U^\rig}  \to D(X)^\rig|_{U^\rig} 
\]
that respects the filtrations, and this morphism is induced by a $\kk$-linear inclusion $V_0 \subset V$.
In particular, 
\[
( \overline{\epsilon}V_0) \otimes _{\breve{\kk}} \co_{U^\rig}  \subset  \overline{\epsilon}  \cdot \Fil (X)^\rig|_{U^\rig} \subset 
 ( \overline{\epsilon}V ) \otimes_{\breve{\kk}} \co_{U^\rig},
\]
and so the restriction of (\ref{periods}) to $U^\rig \subset M^\rig$ takes values in the closed rigid analytic subspace of $N^\rig$  parametrizing codimension one subspaces of $\overline{\epsilon}V$ that contain the line  $\overline{\epsilon}V_0$.  
This contradicts (\ref{periods}) being  \'etale.
 \end{proof}

If $\kk/\Q_p$ is unramified, it is proved in \cite{KR-unitary-1} that every Kudla-Rapoport divisor $Z(x)$ is flat over $\breve{\co}_\kk$.  In Appendix \ref{s:exceptional} we will explain why this is false when $\kk/\Q_p$ is ramified.  


\section{Linear invariance of tensor products}
\label{s:main theorem}


Suppose  $x\in V$ is nonzero.
As in the introduction,  let $I_{Z(x)} \subset \co_M$ be the ideal sheaf defining the Kudla-Rapoport divisor $Z(x) \subset M$, and define a  complex of locally free $\co_M$-modules
\[
C(x) = (   \cdots \to 0 \to I_{Z(x)}  \to \co_M\to 0  )
\]
supported in degrees $1$ and $0$.  We extend the definitions to $x=0$ by setting $Z(0) = M$, and 
\[
C(0) = (   \cdots \to 0 \to \omega  \map{0} \co_M\to 0  )
\]
where $\omega$ is the line bundle of Definition \ref{def:modular forms}.

\begin{theorem}\label{thm:linear invariance}
Fix an $r \ge 0$, and  suppose $x_1,\ldots, x_r \in V$ and $y_1,\ldots,y_r \in V$
generate the same $\co_\kk$-submodule.
For every $i\ge 0$ there is an isomorphism of  coherent  $\co_M$-modules 
\begin{equation}\label{homology win}
 H_i ( C(x_1) \otimes \cdots \otimes C(x_r) )   \iso  H_i ( C(y_1) \otimes \cdots \otimes C(y_r) )  .
\end{equation}
\end{theorem}

\begin{proof}
It is an exercise in linear algebra to check that the list $x_1,\ldots, x_r$ can be transformed to the list $y_1,\ldots, y_r$ using a sequence of elementary operations: 
permute the vectors in the list, scale a vector by an element of $\co_\kk^\times$, 
and add an $\co_\kk$-multiple of one  vector to another. 
The isomorphism class of the complex $ C(x_1) \otimes\cdots \otimes C(x_r)$   is obviously invariant under the first two operations,  and using this one immediately reduces to the case in which
\begin{align*}
y_1& =x_1+ a  x_2 \\
y_2&=x_2 \\
&\, \   \vdots \\
 y_r&=x_r
\end{align*}
for some $a\in \co_\kk$.

Denote by  $Z \subset M$ the closed formal subscheme 
\[
Z(x_1) \cap \cdots \cap Z(x_r) = Z(y_1) \cap \cdots \cap Z(y_r)
\]
(here and below, we use $\cap$ as a shorthand for $\times_M$)  and by $Z\subset \widetilde{Z}$  its first order infinitesimal neighborhood in $M$. 
   Note that both sides of (\ref{homology win}) are supported on $Z$ in the strong sense: 
   they are annihilated by the ideal sheaf defining $Z$.
   
For every $1\le i \le r$,  define  sections
\begin{align*}
s_i & \in H^0 \big( \widetilde{Z}(x_i) , \omega^{-1}|_{ \widetilde{Z}(x_i) } \big) \\
 t_i  & \in H^0 \big( \widetilde{Z}(y_i) , \omega^{-1}|_{ \widetilde{Z}(y_i)  }  \big) 
\end{align*}
by (recall Definition \ref{def:obstruction})
\begin{align*}
s_i & = \begin{cases}
\mathrm{obst}(x_i) & \mbox{if } x_i \neq 0 \\
0 & \mbox{if }x_i=0
\end{cases} \\
t_i & = \begin{cases}
\mathrm{obst}(y_i) & \mbox{if } y_i \neq 0 \\
0 & \mbox{if }y_i=0.
\end{cases} 
\end{align*}
Thus the zero loci of $s_i$ and $t_i$ are  $Z(x_i)$ and $Z(y_i)$, respectively.
After restriction to
\[
\widetilde{Z} \subset \widetilde{Z}(x_1) \cap \cdots \cap  \widetilde{Z}(x_r)\cap  \widetilde{Z}(y_1)\cap \cdots \cap  \widetilde{Z}(y_r)
\]
these sections satisfy
\[
t_1 = s_1 + a s_2,
\]
and $t_i=s_i$ when $i>1$. 
We will approximate $s_1$, $s_2$, and $t_1$, in a noncanonical way, by sections defined over open subsets of $M$.

\begin{lemma}\label{lem:obstruction lifts}
Around every point $z\in Z$ one can find an open affine  neighborhood $U=\Spec(R) \subset M$ over which $\omega_U$ is trivial, and sections
\begin{equation}\label{obstruction lifts}
\sigma_1,  \sigma_2   \in H^0( U, \omega_U^{-1}) \quad\mbox{and}\quad 
 \alpha  \in H^0(U,\co_U) 
\end{equation}
such that the following assertions hold:
\begin{enumerate}
\item[(i)]
$\sigma_1$ has zero locus $Z(x_1)_U$ and   agrees with $s_1$ on $\widetilde{Z}(x_1)_U $,
\item[(ii)]
$\sigma_2$ has zero locus $Z(x_2)_U$ and   agrees with $s_2 $ on $\widetilde{Z}(x_2)_U $,
\item[(iii)]
$\alpha$ restricts to the constant function $a$ on $Z(x_2)_U$,
\item[(iv)]
the section  \[\tau_1 \define \sigma_1 + \alpha \sigma_2\] has zero locus $Z(y_1)_U$
and agrees with $t_1$ on the closed formal subscheme, lying between $Z(y_1)_U$ and $\widetilde{Z}(y_1)_U$, defined by the ideal sheaf 
\[
I_{Z(y_1)_U} \cdot \big(  I_{ Z(y_1)_U } +   I_{ Z( x_2)_U} \big)    \subset \co_U.
\]
\end{enumerate}
Given another  collection of sections
\begin{equation}\label{fake lifts}
\sigma'_1,  \sigma'_2   \in H^0( U, \omega_U^{-1}) \quad\mbox{and}\quad 
 \alpha'  \in H^0(U,\co_U) 
\end{equation}
 satisfying the same properties,  there is an element $\xi \in \mathrm{Frac}(R)$ such that
\begin{equation}\label{homotopy correction}
\xi\cdot  \sigma_1 \otimes \sigma_1'= \tau_1  \otimes \sigma_1' -   \tau_1'   \otimes   \sigma_1 
\end{equation}
and $\xi \cdot I_{Z(x_1)_U}     \subset I_{Z(y_1)_U}  \cdot I_{Z(x_2)_U} $.
\end{lemma}

\begin{proof}
Start with any connected affine open neighborhood $U=\mathrm{Spf}(R)$ of $z\in U$  over which $\omega_U\iso \co_U$, and fix such an isomorphism.  Write  
\begin{align*}
Z(x_1)_U  & =\mathrm{Spf}(R/I_{x_1}) \\
Z(x_2)_U  & =\mathrm{Spf}(R/I_{x_2}) \\
Z(y_1)_U & = \mathrm{Spf}(R/I_{y_1})
\end{align*}
for  ideals $I_{x_1}, I_{x_2}, I_{y_1}  \subset R$,  all  of which 
 are contained in the maximal ideal $\mathfrak{p} \subset R$ determined by the point $z\in U$.
 Identify the sections $s_1$, $s_2$, and $t_1$ with $R$-module generators
\[
s_1 \in I_{x_1} / I_{x_1}^2 , \quad  s_2 \in I_{x_2} / I_{x_2}^2 ,\quad  t_1 \in I_{y_1}/ I_{y_1}^2  , 
\]

Next choose, for $i\in \{1,2\}$,  an arbitrary lift $\sigma_i\in I_{x_i}$ of   $s_i$.
Nakayama's lemma implies $R_\mathfrak{p} \sigma_i = R_\mathfrak{p} I_i$, and so
 there is some  $f \not\in \mathfrak{p}$  such that  $R[1/f]   \sigma_i = R[1/f]  I_i$. After inverting $f$, and hence shrinking $U$, we may  assume that  $R \sigma_i = I_i$.   We now have sections $\sigma_1$ and $\sigma_2$ satisfying properties (i) and (ii).

Choose  an arbitrary lift $\hat{\tau}_1\in I_{y_1}$ of $t_1$.
Again using Nakayama's lemma, we may   shrink $U$ in order to assume that
$
R\hat{\tau}_1 = I_{y_1}.
$
The relation  $y_1=x_1+ a x_2$ implies the equality 
\begin{equation}\label{wee physical}
Z(y_1) \cap Z(x_2) = Z(x_1)\cap Z(x_2)
\end{equation}
of closed formal subschemes of $M$, and hence
\begin{equation}\label{wee intersection}
I_{y_1}+I_{x_2}   =I_{x_1}+I_{x_2} .
\end{equation}
Along the first order infinitesimal neighborhood of (\ref{wee physical}) in $M$  we have
$
t_1 = s_1+ a s_2.
$
This implies that $\hat{\tau}_1 \equiv \sigma_1+a \sigma_2$   modulo the square of (\ref{wee intersection}), 
and so we may write
\[
\hat{\tau}_1 = \sigma_1 + a \sigma_2 + ( A \hat{\tau}_1^2+ B   \hat{\tau}_1\sigma_2+ C \sigma_2^2 )
\]
for some $A,B,C \in R$. Now rewrite this as
\[
 \tau_1 = \sigma_1 + \alpha \sigma_2
\]
where $\tau_1 = \hat{\tau}_1-A\hat{\tau}_1^2 - B \hat{\tau}_1\sigma_2$ and $\alpha= a + C \sigma_2$.

By construction $\tau_1$ agrees with $\hat{\tau}_1$, hence also with $t_1$,   in $R/I_{y_1}(I_{y_1}+I_{x_2})$.
In particular it generates $I_{y_1}/\mathfrak{p}I_{y_1}$ as an $R$-module, and the above argument 
using  Nakayama's lemma allows us to shrink $U$ in order to assume that  $R\tau_1 = I_{y_1}$.
  The sections $\sigma_1$, $\sigma_2$, and $\alpha$ we have constructed  satisfy properties (i), (ii), (iii), and (iv).

Now suppose we have another collection of sections (\ref{fake lifts}) satisfying the same properties.
As above, we use $\omega_U\iso \co_U$ to identify
$\sigma'_1,  \sigma'_2 , \alpha'   \in R$, so that 
\[
R \sigma_1 = I_{x_1} = R\sigma_1' ,\quad R \sigma_2 = I_{x_2} = R\sigma_2' ,\quad R \tau_1 = I_{y_1} = R\tau_1' . 
\]

In the degenerate case where  $I_{x_1}=0$ (this can only happen when $x_1=0$) we must have $\sigma_1=0=\sigma_1'$, and any choice of  $\xi \in R$ will satisfy the stated properties. Thus we may assume  $I_{x_1}\neq 0$.

Define $\xi \in \mathrm{Frac}(R)$ by 
\[
\xi =  \left(  \frac{\tau_1}{\sigma_1} - \frac{\tau_1'}{\sigma_1'}  \right)
=  \left( \frac{ \alpha \sigma_2}{\sigma_1} - \frac{\alpha' \sigma_2'}{ \sigma_1'}  \right)  .
\]
We need to show that  $ R\xi \sigma_1 \subset  R \tau_1\sigma_2$.  
As $R$ is  regular, it is equal to the intersection of  its localizations  at height one primes $\mathfrak{q}\subset R$, and every such  localization $R_\mathfrak{q}$  is a DVR.
Thus  it suffices to prove, for all such $\mathfrak{q}$, 
\begin{equation}\label{height one}
 \ord_\mathfrak{q}(\xi  \sigma_1)  \ge \ord_\mathfrak{q}( \tau_1 \sigma_2 ) .
\end{equation}

The conditions imposed on our sections imply the congruences
\begin{align*}
\sigma_1 \equiv s_1 \equiv \sigma_1'  &  \pmod{I_{x_1}^2} \\ 
\alpha \sigma_2 \equiv a s_2 \equiv \alpha'\sigma'_2  &  \pmod{ I_{x_2}^2 } \\
\tau_1 \equiv t_1 \equiv \tau_1'   & \pmod{  I_{y_1}(I_{y_1}+I_{x_2})   } ,
\end{align*}
the first and third of which imply
\begin{align}
 \sigma_1 / \sigma_1'  & \equiv 1 \pmod { R\sigma_1  }  \label{unit congruence} \\
 \tau_1 / \tau_1' & \equiv 1 \pmod { R\tau_1 + R\sigma_2  }. \nonumber
\end{align}

First assume  $\ord_\mathfrak{q}( \sigma_2) \ge \ord_\mathfrak{q}(\tau_1)$, and note that $\tau_1= \sigma_1 + \alpha \sigma_2$ implies
\[
\ord_\mathfrak{q}( \sigma_1)  \ge \mathrm{min} \{ \ord_\mathfrak{q}(\tau_1) , \ord_\mathfrak{q}( \alpha \sigma_2 )  \}  =  \ord_\mathfrak{q}( \tau_1).
\]
 It  follows from this and (\ref{unit congruence}) that $\sigma_1 / \sigma_1' \equiv 1 \pmod {  R_\mathfrak{q} \tau_1  }$, and hence
\[
\xi  \sigma_1  =   \alpha \sigma_2 - \frac{\sigma_1}{\sigma'_1} \cdot  \alpha' \sigma_2' 
 \equiv  \alpha \sigma_2 \left( 1  - \frac{\sigma_1}{\sigma'_1} \right)  \pmod{  R_\mathfrak{q} \sigma_2^2   }  .
\]
This implies $ \xi  \sigma_1 \equiv 0  \pmod{ R_\mathfrak{q} \tau_1 \sigma_2   }$, proving (\ref{height one}).

Now assume $\ord_\mathfrak{q}(  \sigma_2 ) <  \ord_\mathfrak{q}( \tau_1 )$.   The relation $\tau_1=\sigma_1+\alpha \sigma_2$ implies 
\[
 \ord( \sigma_1) \ge \mathrm{min}\{    \ord( \tau_1),  \ord( \alpha \sigma_2)    \}  \ge \ord_\mathfrak{q}(  \sigma_2 ),
\]
and also  $R_\mathfrak{q} \tau_1 +  R_\mathfrak{q}  \sigma_2= R_\mathfrak{q} \sigma_2.$
The congruences of  (\ref{unit congruence}) therefore  imply
\begin{align*}
 \sigma_1 /  \sigma'_1   & \equiv 1 \pmod{   R_\mathfrak{q}  \sigma_2} \\
\tau_1' /  \tau_1  &   \equiv 1 \pmod{  R_\mathfrak{q} \sigma_2 }  ,
\end{align*}
and hence
\[
\xi  \sigma_1  =  \tau_1 \left( 1  - \frac{\sigma_1}{\sigma'_1} \frac{\tau'_1}{\tau_1} \right)   \equiv 0  \pmod{    R_\mathfrak{q} \tau_1\sigma_2 }.
\]
Once again, this proves (\ref{height one}).
\end{proof}

For a fixed  $z\in Z$, choose an open neighborhood   $U\subset M$   and  sections (\ref{obstruction lifts})  as in Lemma \ref{lem:obstruction lifts}.

\begin{lemma}\label{lem:wee complex}
The choice of sections  (\ref{obstruction lifts}) determines an isomorphism
\begin{equation}\label{wee complex}
  f : C(x_1)_U  \otimes C(x_2)_U  \iso  C(y_1)_U  \otimes C(x_2)_U, 
\end{equation}
and changing the sections  changes the isomorphism by a homotopy.
\end{lemma}

\begin{proof}
The choice of sections determines complexes of locally free $\co_U$-modules
\begin{align*}
D(x_1) & = (   \cdots \to 0 \to \omega_U \map{ \sigma_1 } \co_U\to 0  )\\
D(x_2) & = (   \cdots \to 0 \to \omega_U \map{ \sigma_2 } \co_U\to 0  ) \\
D(y_1) & = (   \cdots \to 0 \to \omega_U \map{ \tau_1 } \co_U\to 0 ),
\end{align*}
and there are obvious isomorphisms
\[
 D(x_1)  \iso C(x_1)_U ,\quad 
  D(x_2) \iso C(x_2)_U ,\quad 
 D(y_1)   \iso C(y_1)_U.
\]
Indeed, if  $x_1\neq 0$ then 
\begin{equation*}
\xymatrix{
{\cdots}\ar[r]  & 0 \ar[r] \ar@{=}[d] & {\omega_U} \ar[r]^{ \sigma_1}  \ar[d]^{ \sigma_1} & { \co_U } \ar[r] \ar@{=}[d]& 0  \\
{\cdots}\ar[r]   & 0 \ar[r] & { I_{Z(x_1)_U } }  \ar[r] & { \co_U } \ar[r]& 0
}
\end{equation*}
defines  an isomorphism $D(x_1)\iso  C(x_1)_U$.   On the other hand, if $x_1=0$ then $\sigma_1=0$, and $D(x_1)=C(x_1)_U$ by definition. The other isomorphisms are entirely similar.

To define $f$, it now suffices to define an isomorphism
\[
g : D(x_1)  \otimes D(x_2)  \iso  D(y_1)  \otimes D(x_2) .
\]
The complexes in question are 
\begin{align*}
D(x_1)  \otimes D(x_2) 
& = ( \cdots \to 0 \to \omega_U \otimes  \omega_U  \map{ \partial_2 } \omega_U\oplus \omega_U \map{ \partial_1 }  \co_U \to 0 )  \\
D(y_1)  \otimes D(x_2) 
& = ( \cdots \to 0 \to \omega_U \otimes  \omega_U  \map{ \partial^*_2 } \omega_U\oplus \omega_U \map{ \partial^*_1 }  \co_U \to 0 ) 
\end{align*}
where the boundary maps are defined by
\begin{align*}
 \partial_1  (\eta_1,\eta_2)  & =  \sigma_1(\eta_1) + \sigma_2 (\eta_2) \\
 \partial^*_1  (\eta_1,\eta_2)  & =  \tau_1(\eta_1) + \sigma_2 (\eta_2) \\
 \partial_2 (    \eta_1 \otimes \eta_2   ) &=  \big(   \sigma_2 (-\eta_2) \eta_1  , \sigma_1 (\eta_1 )\eta_2 \big) \\
 \partial^*_2 (    \eta_1 \otimes \eta_2   ) &=  \big(   \sigma_2 (-\eta_2) \eta_1  , \tau_1 (\eta_1 )\eta_2 \big)
 \end{align*}
for local sections $\eta_1$  and $\eta_2$ of $\omega_U$.  
Recalling that $ \tau_1   =    \sigma_1   +     \alpha  \sigma_2$,  the desired isomorphism is
\[
\xymatrix{
{ \cdots }\ar[r]    &  {  0  } \ar[r] \ar@{=}[d] &    {   \omega_U\otimes\omega_U } \ar[rr]^{\partial_2  }  \ar@{=}[d] &  &   { \omega_U \oplus \omega_U }  \ar[rr]^{\partial_1}  \ar[d]^{ g_1  } &  & {\co_{U} }\ar[r]\ar@{=}[d] & 0 \\
{ \cdots }\ar[r]    &  {  0  } \ar[r]  &    {   \omega_U\otimes\omega_U } \ar[rr]^{\partial^*_2 }  &  &   { \omega_U \oplus \omega_U }  \ar[rr]^{\partial^*_1}  &  & {\co_{U}}\ar[r] & 0,
}
\]
where  $g_1( \eta_1,\eta_2) =  (   \eta_1 ,  \eta_2 -  \alpha   \eta_1  )$.

Having constructed the isomorphism (\ref{wee complex}), we now study its dependence on the sections (\ref{obstruction lifts}).  
Suppose we have another  collection of sections (\ref{fake lifts}), and hence two isomorphisms
\[
f,f' : C(x_1)_U  \otimes C(x_2)_U  \iso  C(y_1)_U  \otimes C(x_2)_U.
\]
We must prove that $f$ and $f'$ are  homotopic.

If $x_2=0$ then $y_1=x_1$, and the conditions imposed on the sections (\ref{obstruction lifts}) imply that $\sigma_1=\tau_1$,  $\sigma_2=0$, and $\alpha=a$.
From this it is easy to see that $f=f'$, and so henceforth we assume that $x_2\neq 0$.

If $x_1=0$ and $y_1=0$, then  the conditions imposed on  (\ref{obstruction lifts}) imply that $\sigma_1=0$ and $\tau_1=0$.  
The relation $\tau_1=\sigma_1+\alpha\sigma_2$ and our assumption $x_2\neq 0$ therefore imply that $\alpha=0$.
Tracing through the definitions, we again find that $f=f'$.

If  $x_1\neq 0$ and $y_1\neq 0$ then  $ e=f-f'$ is given explicitly by
\[
\xymatrix{
 {  0  } \ar[r]  &    {   I_{Z(x_1)_U} \otimes   I_{Z(x_2)_U}    } \ar[r]^{\partial_2  }  \ar[d]^{e_2} &     {I_{Z(x_1)_U} \oplus   I_{Z(x_2)_U}   }  \ar[rr]^{\partial_1}  \ar[d]^{ e_1 }   \ar@{-->}[dl]_{ h_1}  & & {\co_{U} }\ar[r]\ar[d]^0  \ar@{-->}[dll]_{ h_0} & 0 \\
  {  0  } \ar[r]  &    {     I_{Z(y_1)_U} \otimes   I_{Z(x_2)_U}    } \ar[r]^{\partial^*_2 }    &   {  I_{Z(y_1)_U} \oplus   I_{Z(x_2)_U} }  \ar[rr]^{\partial^*_1}  &  & {\co_{U}}\ar[r] & 0,
}
\]
where the boundary maps are 
\begin{align*}
 \partial_1  (\eta_1,\eta_2)  & = \eta_1+\eta_2  \\
 \partial^*_1  (\eta_1,\eta_2)  & =  \eta_1+\eta_2   \\
 \partial_2 (    \eta_1 \otimes \eta_2   ) &=   (   -\eta_1  \eta_2  ,  \eta_1  \eta_2  ) \\
 \partial^*_2 (    \eta_1 \otimes \eta_2   ) &=   (   -\eta_1 \eta_2  , \eta_1 \eta_2  ),
 \end{align*}
and 
\[
e_1( \eta_1,\eta_2) =   (  \xi \eta_1  ,  -\xi \eta_1  ) ,\qquad 
e_2( \eta_1\otimes \eta_2) =  \xi \eta_1\otimes \eta_2  .
\] 
Here $\xi$ is the rational function on $U$ of Lemma \ref{lem:obstruction lifts}
($\xi$ is uniquely determined by the relation (\ref{homotopy correction}), as our assumption $x_1\neq 0$ implies that $\sigma_1$ and $\sigma'_1$ are nonzero).
The dotted arrows, which exhibit the homotopy between $e$ and $0$, are defined by 
$h_0( \eta ) = (0,0)$ and $h_1(\eta_1,\eta_2) = -\xi \eta_1 \cdot 1\otimes 1$.
Note that the definition of $h_1$ only makes sense because of the inclusion 
\[
\xi \cdot I_{Z(x_1)_U}     \subset I_{Z(y_1)_U}  \cdot I_{Z(x_2)_U} 
\]
 of Lemma \ref{lem:obstruction lifts}.

If $x_1= 0$ and $y_1\neq 0$ then $ e=f-f'$ is given explicitly by 
\[
\xymatrix{
 {  0  } \ar[r]  &    {   \omega_U \otimes   I_{Z(x_2)_U}    } \ar[r]^{\partial_2  }  \ar[d]^{e_2} &     { \omega_U \oplus   I_{Z(x_2)_U}   }  \ar[rr]^{\partial_1}  \ar[d]^{ e_1 }   \ar@{-->}[dl]_{ h_1}  & & {\co_{U} }\ar[r]\ar[d]^0  \ar@{-->}[dll]_{ h_0} & 0 \\
  {  0  } \ar[r]  &    {     I_{Z(y_1)_U} \otimes   I_{Z(x_2)_U}    } \ar[r]^{\partial^*_2 }    &   {  I_{Z(y_1)_U} \oplus   I_{Z(x_2)_U} }  \ar[rr]^{\partial^*_1}  &  & {\co_{U}}\ar[r] & 0,
}
\]
where the boundary maps are 
\begin{align*}
 \partial_1  (\eta_1,\eta_2)  & =   \eta_2  \\
 \partial^*_1  (\eta_1,\eta_2)  & =  \eta_1+  \eta_2   \\
 \partial_2 (    \eta_1 \otimes \eta_2   ) &=   (   -\eta_2  \eta_1  ,  0  ) \\
 \partial^*_2 (    \eta_1 \otimes \eta_2   ) &=  (   -\eta_2 \eta_1  ,  \eta_2  \eta_1 ),
 \end{align*}
and, setting $\zeta = \tau_1-\tau_1'  \in H^0(U, \omega_U^{-1})$, 
\[
e_1( \eta_1,\eta_2) = \left(  \zeta( \eta_1 )   ,  -\zeta( \eta_1)    \right)   ,\qquad 
e_2( \eta_1\otimes \eta_2) =  \zeta(  \eta_1) \otimes \eta_2 .
\] 
The dotted arrows, which exhibit the homotopy between $e$ and $0$, are defined by 
$h_0(\eta) = (0,0)$ and $h_1(\eta_1,\eta_2) = -\zeta (\eta_1) \cdot 1\otimes 1$.
To make sense of the definition of $h_1$, note  that the relation $y_1=ax_2$ implies $Z(x_2) \subset Z(y_1)$, and hence
\[
\zeta(\eta_1) \in   I_{Z(y_1)_U} \cdot   (     I_{Z(y_1)_U} +    I_{Z(x_2)_U}  ) =  I_{Z(y_1)_U} \cdot   I_{Z(x_2)_U}.
\]

The case $x_1\neq 0$ and $y_1=0$ is entirely analogous to the previous case, and we leave the details to the reader.
\end{proof}

As $y_j=x_j$ for $j\ge 2$,  the  isomorphism of Lemma \ref{wee complex}  determines an isomorphism
\begin{equation*}
C(x_1)_U  \otimes\cdots \otimes  C(x_r)_U  \iso C(y_1)_U  \otimes \cdots \otimes  C(y_r)_U ,
\end{equation*}
whose homotopy class does not depend on the choices  (\ref{obstruction lifts})  used in its construction.
Hence the  induced  isomorphism 
\[
H_i ( C(x_1)  \otimes\cdots \otimes  C(x_r) )_U  \iso  H_i ( C(y_1)  \otimes \cdots \otimes  C(y_r) )_U
\]
of $\co_U$-modules  also does not depend on these choices.
By varying $U$  and gluing, we  obtain an isomorphism (\ref{homology win}) 
defined over an open neighborhood of $Z$ in $M$.
We have already noted that both sides of (\ref{homology win}) are  supported on $Z$, and so the isomorphism extends uniquely to all of $M$.  This completes the proof of Theorem  \ref{thm:linear invariance}.
\end{proof}

\appendix
\section{The exceptional divisor}
\label{s:exceptional}
%

Throughout this appendix we assume that $\kk/\Q_p$ is ramified.  
We want to explain why the Kudla-Rapoport divisors of Definition \ref{def:KR divisor} are generally not flat over $\breve{\co}_\kk$.

Denote by $\breve{\F} = \breve{\co}_\kk/ \breve{\mathfrak{m}}$ the residue field of $\breve{\co_\kk}$.  
The two embeddings  $\varphi,\overline{\varphi} : \co_\kk \to \breve{\co}_\kk$  necessarily reduce to the unique $\Z_p$-algebra morphism $\co_\kk \to \breve{\F}$

\begin{definition}
The \emph{exceptional divisor} $\Exc \subset M$ is the set of all points $s\in M$ at which the action  
\[
i : \co_\kk \to \End( \Lie(X_s) )
\]
 is through scalars; that is to say, the action factors through the unique   morphism $\co_\kk \to \breve{\F}$.
This is a closed subset of the underlying topological space of $M$, and we endow it with its induced structure of a reduced scheme over $\breve{\F}$.
\end{definition}

\begin{proposition}\label{prop:exceptional}
 The exceptional divisor $\Exc \subset M$ is a Cartier divisor, and  is isomorphic to a disjoint union of copies of the projective space $\mathbb{P}^{n-1}$ over $\breve{\F}$.
\end{proposition}

\begin{proof}
 A point  $s\in M(\breve{\F})$ corresponds to a pair $(X_{0s},X_s)$ over $\breve{\F}$,  
which we recall is really a tuple   
\[
(X_{0s},i_0, \lambda_0,\varrho_0, X_s, i ,\lambda, F_{X_s}  ,\varrho)\in M(\breve{\F}).
\]
 If $s\in \Exc(\breve{\F})$ then the action of $\co_\kk$ on $\Lie(X)$ is through the unique $\Z_p$-algebra morphism $\co_\kk \to \breve{\F}$.  
 This implies that \underline{any} codimension one subspace of $F\subset \Lie(X_s)$ satisfies Kramer's signature condition as in \S \ref{s:RZ}, and we obtain a closed immersion
\[
\mathbb{P}( \Lie(X_s)^\vee ) \hookrightarrow \Exc 
\]
by sending $F \mapsto (X_{0s},i_0, \lambda_0,\varrho_0, X_s, i ,\lambda, F   ,\varrho).$
In other words, vary the codimension one subspace in $\Lie(X_s)$ and leave all other data fixed.

It is clear that $\Exc$ is the disjoint union of all such closed subschemes, and that every connected component of $\Exc$ is reduced, irreducible,  and of codimension one in $M$.  The regularity of $M$ then implies that $\Exc\subset M$ is defined locally by one equation.
\end{proof}

\begin{proposition}\label{prop:lots of exceptional}
Fix a nonzero $x\in V$, and any  connected component $D \subset \Exc$.
For all $k\gg 0$ we have 
$
D \subset Z(p^k x).
$
  In particular $Z(p^k x)$ is not flat over $\breve{\co}_\kk$.
\end{proposition}

\begin{proof}
If we fix one point $s\in D$,  Lemma \ref{lem:hermitian inclusion} allows us to view
\[
x\in \Hom_{\co_\kk} (X_{0s} , X_s) [ 1/p] .
\]
For all $k\gg 0$ we thus have 
$
p^kx\in \Hom_{\co_\kk} (X_{0s} , X_s) .
$
It follows from the characterization of $D \iso \mathbb{P}^{n-1}$ found in the proof of Proposition \ref{prop:exceptional} that the $p$-divisible groups $X_{0D}$ and $X_D$ are constant (that is, are pullbacks via $D\to \Spec(\breve{\F})$ of $p$-divisible groups over $\breve{\F}$), and hence the restriction map
\[
\Hom_{\co_\kk} (X_{0D} , X_D) \to \Hom_{\co_\kk} (X_{0s} , X_s)
\]
is an isomorphism. Hence $p^k x\in \Hom_{\co_\kk} (X_{0D} , X_D)$ and  $D \subset Z(p^k x)$.
\end{proof}

\bibliographystyle{amsalpha}

\newcommand{\etalchar}[1]{$^{#1}$}
\providecommand{\bysame}{\leavevmode\hbox to3em{\hrulefill}\thinspace}
\providecommand{\MR}{\relax\ifhmode\unskip\space\fi MR }
\providecommand{\MRhref}[2]{%
  \href{http://www.ams.org/mathscinet-getitem?mr=#1}{#2}
}
\providecommand{\href}[2]{#2}

\end{document}